\documentclass[11pt]{amsart}
\usepackage{pdfsync}
\usepackage[all]{xy}
\usepackage{tikz}
\usepackage{tikz-cd}

\usepackage{amsmath,amssymb,amscd,amsfonts,verbatim}
\usepackage{paralist}
\usepackage[mathscr]{eucal}

\usepackage{enumerate}

\usepackage[pdftex]{hyperref}
\hypersetup{citecolor=blue,linktocpage}

\newtheorem{thm}{Theorem}[section]
\newtheorem{lem}[thm]{Lemma}
\newtheorem{cor}[thm]{Corollary}
\theoremstyle{definition}

\newtheorem{assu-nota}[thm]{Assumption--Notation}
\newtheorem{ex}[thm]{Example}

\newcommand{\ol}{\overline}

\newcommand{\K}{\mathbb K}
\newcommand{\N}{\mathbb N}

\newcommand{\Dc}{\mathcal D}

\newcommand{\Lc}{\mathcal L}

\DeclareMathOperator{\IM}{Im}

\DeclareMathOperator{\Spec}{Spec}

\newcommand{\OO}{\mathcal{O}}

\usepackage[colorinlistoftodos]{todonotes}

\newcommand\rpi[1]{\todo[inline,color=yellow]{#1}}

\newcommand\ritai[1]{\todo[inline,color=pink]{#1}}

\usepackage{todonotes}

\numberwithin{equation}{section}
\title[Parity of  theta characteristics and infinitesimal deformations]{The parity of  theta characteristics is preserved by infinitesimal deformations}
 
 \thanks   {\tiny {Research partially supported by Funda\c c\~ao para a Ci\^encia e Tecnologia (FCT), Portugal, through CAMGSD, IST-ID, projects UID/4459/2025, UIDB/04459/2020 and  UIDP/04459/2020, by the European Union - Next Generation EU, Mission 4 Component 2 - CUP E53D23005400001 and by  PRIN 2022BTA242 ``Geometry of algebraic structures: Moduli, Invariants, Deformation'' of Italian MUR.
Part of this project was carried out while the authors were guests of the Research in Pairs program of CIRM-Trento.
The first named author is a member of Centro de An\'alise Matem\'atica, Geometria e Sistemas Din\^amicos.
 The second and third named authors are  members of GNSAGA of INDAM. \newline
 We are indebted to Barbara Fantechi for explaining us how to use Cohomology and Base Change in this context. 
 }}

\author{ Margarida Mendes Lopes} 
\author{ Rita Pardini} 
\author{Roberto Pignatelli}

\begin{document}
\begin{abstract}  
In this note, given a family of  relative dimension one over a smooth curve, we determine the parity of the restriction of a relative theta characteristic to an arbitrary multiple of a fiber in terms of the parity of the restriction to a general fibre.\\
This result can be regarded as a variant of the well-known theorem on the invariance of the parity of theta characteristics in families.\\
As a corollary, we obtain that the torsion subsheaf of the first higher direct image sheaf of a relative theta characteristic splits as a direct sum of two isomorphic sheaves.\\


 
 \end{abstract}

\maketitle

 \setcounter{tocdepth}{1}
 
\tableofcontents

\section{Introduction}

We work over an algebraically closed field $\mathbb K$ of characteristic $\neq 2$.\par

In  \cite{Mum71} Mumford has given an algebraic proof, independent of the theory of theta functions, of the classically known fact that the parity of theta-characteristics is constant in families. We recall below  a   generalized version  of this result, due to Harris (see also \cite{Cor89}):
\smallskip

\noindent {\bf Theorem.} (\cite[Theorem 1.10.(i)]{Har}) 
{\em  Let $\Delta$ be an irreducible  variety,  let $\pi \colon X\rightarrow \Delta$ be a proper flat map with fibers $C_t:=\pi^{-1}(t)$ reduced curves,  let   $\mathcal L$ be  a line bundle  on $X$  and set $\mathcal L_t:={\mathcal L}|_{C_t}$.
 
 If  ${\mathcal L}_t^{\otimes 2}
\cong \omega_{C_t}$ for all  $t\in \Delta$,  then the function $t\mapsto h^0(C_t,{\mathcal L}_t)$ is constant modulo 2.}
\bigskip

In this note we  prove an infinitesimal version of the above result:

\begin{thm}\label{thm: infinitesimal}
Let $\Delta$ be  a smooth  connected curve and let $\pi\colon X\to \Delta$ a  projective flat morphism  whose fibers  $C_t:=\pi^{-1}(t)$ are reduced connected curves and let   $\mathcal L$ be  a line bundle  on $X$.

If   ${\mathcal L}_t^{\otimes 2}
\cong \omega_{C_t}$ for all  $t\in \Delta$, then   $h^0(kC_t,{\mathcal L}|_{kC_t})=kh^0(C_t, \mathcal L_t)$ modulo 2 for all $k\in \N_{>0}$ and for all $t\in \Delta$.

Fixed $t\in \Delta$, the even numbers $kh^0(C_t, \mathcal L_t)-h^0(kC_t,{\mathcal L}|_{kC_t})$ form a non-decreasing sequence, indexed by $k$.
\end{thm}

The last sentence may be seen as an infinitesimal version of semicontinuity.
Combining the two previous theorems, we obtain:
\begin{cor} \label{cor: torsion}
In the assumptions of Theorem \ref{thm: infinitesimal},      there is a  coherent  sheaf $\mathcal T$ on $\Delta$ such that the torsion subsheaf of $R^1\pi_*\Lc$ is isomorphic to $\mathcal T\oplus \mathcal T$. 
\end{cor}
The proofs are given in the next section.
\bigskip

Our interest in this question arose in studying surfaces of general type with canonical map of odd degree (cf. \cite{odd-degree}). The fact that the parity of theta characteristics is constant in families is crucial throughout our analysis of such surfaces but it does not suffice to deal with some of the possible cases, that we finally managed to rule out by means of Corollary \ref{cor: torsion}.

We think however that Theorem \ref{thm: infinitesimal} and Corollary \ref{cor: torsion} are of independent interest. 
Actually, as suggested by J. Koll\'ar, whom we thank warmly, it is natural to ask whether Theorem \ref{thm: infinitesimal} can be generalized to the case where $B$ is any local artinian $\K$-algebra.  Namely, given a flat family of curves $\mathcal X\to \Spec B$ with closed fiber $X_0$ and a relative theta characteristic $\mathcal L$ on $\mathcal X$, is it true that the number  $\dim_{\K}B \cdot h^0(\mathcal L|_{X_0})-h^0(\mathcal X, \mathcal L)$ is even?
 
 We do not know the answer to this question; we just observe that our method of proof does not extend to the general situation (cf.  Example \ref{ex: kollar}).

\section{Proofs}

\subsection{Preliminary results}
The proof of Theorem \ref{thm: infinitesimal} uses some  auxiliary results  that we now explain.

Let $\Delta$ be a smooth connected curve, fix $\bar t\in \Delta$,  set $A:=\OO_{\Delta, \bar t}$,  $B_k:=A/s^kA$,  where $s\in A$ a local parameter.  \begin{lem}\label{lem: skew}
 Let $q>0$ be an integer and let $\psi\colon A^q\to A^q$ be an $A$-linear map given by a skew-symmetric matrix $M$; for $k\in \N_{>0}$  let $\psi_k\colon B_k^q\to B_k^q$ be the map induced by $M$ and let $r_k$ be the dimension of $\IM \psi_k$ as a $\K$-vector space.
 
 Then  $\{r_k\}$ is a non-decreasing sequence of even integers.
 \end{lem}
 \begin{proof}
 
 Writing  $M$ as 
  \[
  M=\sum_{j=0}^{k-1} s^j M_j\quad \mod s^k
  \]
  we get skew-symmetric matrices $M_j$ with  coefficients in $\K$.
Denote by $r_k$ the dimension as a $\K$-vector space of the image of $\psi_k$.
Let  $c_1,\ldots,c_q$ be the standard basis of $B_{k}^q$ as a $B_{k}$-module. Then
\begin{equation}\label{eqn:basis of the domain}
c_1,\ldots,c_q, sc_1, \ldots, sc_q, s^2c_1,\ldots, s^{k-1}c_q
\end{equation}
is a basis of $B_{k}^q$ as a $\K$-vector space. We now use this basis to associate a matrix with the operator $\psi_{k}$, but we order it differently when using it as a basis for the domain or as a basis for the codomain. Precisely we order it as in \eqref{eqn:basis of the domain} as a basis of the domain and as 
\begin{equation*}
s^{k-1}c_1,\ldots,s^{k-1}c_q, s^{k-2}c_1, \ldots, s^{k-2}c_q, s^{k-3}c_1,\ldots, c_q
\end{equation*}
as a basis of the codomain. Then a straightforward computation shows that the matrix associated to  $\psi_{k}$ with respect to this choice of the bases is the block matrix
\[ N_k:=
\begin{pmatrix}
M_{k-1} & M_{k-2} & M_{k-3} & \cdots &M_0 \\
M_{k-2} & M_{k-3} & M_{k-4} & \cdots & 0  \\
M_{k-3} & M_{k-4} & M_{k-5} & \cdots & 0 \\
\vdots &  &   & \ddots &  \vdots\\
M_0&0 & 0 &  \cdots & 0 \\
\end{pmatrix}
\]
As all the matrices $M_j$ are skewsymmetric, $N_k$ is skewsymmetric as well, and therefore $r_k$, which equals the rank of $N_k$, is even.

Finally, we note that $N_{k}$ is a submatrix of $N_{k+1}$, and therefore the sequence of even numbers $\{r_k\}_{k \in {\mathbb N}}$  is non-decreasing.  
 \end{proof}
 
Lemma \ref{lem: skew} is  key for the proof of Theorem \ref{thm: infinitesimal}. The following example shows that it does not hold in general for a local  $\K$-algebra $A$ and an  artinian quotient $B=A/I$.
 \begin{ex}\label{ex: kollar}
 Set $A=\K[x,y]$, $I=(x,y)^2$ and $B:=A/I$. Consider the skewsymmetric matrix
 \[ M:=
\begin{pmatrix}
0 & 0 & x \\
0 & 0 & y \\
-x & -y & 0 \\
\end{pmatrix}
\]
and let $\ol\psi \colon B^3\to B^3$ be the map induced by $M$. The image of $\ol \psi$ is spanned as a $\K$-vector space by the independent vectors:
\[
^t(0,0,x),\ \  ^t(0,0,y),\ \  ^t(x,y,0)
\] 
  and therefore it has odd dimension.
   \end{ex}
   \bigskip

  Let $V$ be a free $A$-module of rank $2r$ with a symmetric bilinear  form $Q\colon V\times V\to A$ whose reduction modulo $s$, $\ol Q\colon \ol V\times \ol V\to\K$,   is non degenerate. Assume that $W_1,W_2\subset V$ are free rank $r$ submodules such that $V/W_i$ is free for $i=1,2$. In particular, the map $W_i\otimes_A\!B_k\to V\otimes _A\!B_k$ is injective for $i=1,2$ for all $k$. 
  
  Set now, for all $k \ge 1$ 
  \[
    q_k:=  \dim_{\K}\left( \left(W_1\otimes_A\!B_k\right)\cap \left(W_2\!\otimes_A\!B_{k}\right) \right>)
\]

  Then we get the following infinitesimal version of the permanence of the parity of the dimension of the intersection of two maximal isotropic subspaces:   
  \begin{lem}\label{lem: aux} In the above set-up, if $W_1$ and $W_2$ are totally isotropic for $Q$, then the sequence $\{kq_1-q_k\}_{k \in {\mathbb N}}$ is a non-decreasing sequence of even numbers.  
  \end{lem}
  \begin{proof}
  Given an $A$-module $N$ we write $\ol N:=N/sN$ and  for $z\in N$ we denote by $\bar z\in \ol N$ its image. 
  \smallskip
  
   We start by showing that any basis $e_1,\dots e_r$ of $W_1$ as $A$-module can be completed to a basis $e_1,\dots e_r, f_1,\dots f_r$ of $V$ such that $Q(f_i,f_j)=0$ and $Q(f_i,e_j)=\delta_{ij}$ for all $1\le i,j\le r$. 
   
   The natural map $V\to W_1^{\vee}$ induced by $Q$ is surjective,  since $W_1$ is a direct summand of $V$ and $\ol Q$ is non  degenerate. So we may find $w_1, \dots w_r\in V$ such that $Q(w_i,e_j)=\delta_{ij}$ for all $1\le i,j\le r$ and set $f_i:=w_i-\frac 12\sum_{j=1}^rQ(w_i, w_j)e_j$.
   
   We set for sake of simplicity $q:=q_1$. By definition of  $q_k$, $\dim_{\mathbb K}  \left( \ol W_1\cap \ol W_2 \right) =q$.
   We choose a basis  $e_1,\dots e_r$ of $W_1$ such that $\bar e_1,\dots \bar e_q$ is a basis of $\ol W_1\cap \ol W_2$ and complete it to a basis $e_1,\dots e_r, f_1,\dots f_r$  of $V$ as above.

   One can pick a basis $v_1,\dots v_r$ of $W_2$ such that $\bar v_i=\bar e_i$ for $i=1, \dots q$ and such that $\bar v_i=\bar f_i+ \sum_{j=q+1}^r a_{ij}\bar e_j$ for $i=q+1,\dots r$ and some scalars $a_{ij}\in \K$. Since $\ol W_2$ is a totally isotropic subspace, the matrix $(a_{ij})$ is skewsymmetric. So, replacing $f_i$ by $f_i+ \sum_{j=q+1}^r a_{ij} e_j$ for $i=q+1,\dots r$, we may assume in addition that $\ol v_i=\bar f_i$ for $i=q+1,\dots r$.

   Denote by $U_1\subset V$ the span of $e_1,\dots e_q, f_{q+1},\dots f_r$ and by $U_2$ the span of $f_1,\dots f_q, e_{q+1},\dots e_r$, so that $V=U_1\oplus U_2$ is a decomposition as the sum of totally isotropic subspaces.
    As $\ol W_2= \ol U_1$, by Nakayama's Lemma  the projection $V\to U_1$ with kernel $U_2$  restricts to a surjective map, hence an isomorphism,  $W_2\to U_1$.
      So   we may write 
\begin{align*}
  v_i= e_i+sz_i& \text{ for }i=1, \dots , q&
  v_i= f_i+sz_i& \text{ for }i=q+1,\dots , r&
\end{align*}
where   
\begin{equation}\label{eqn: zi}
z_i=\sum_{j=q+1}^r\lambda_{ij}e_j+\sum_{j=1}^{q}\mu_{ij}f_j
\end{equation}

Then $\left(W_1\otimes_A\!B_k\right)+ \left(W_2\!\otimes_A\!B_k\right)$ is generated, as a  $\K$-vector space, by the classes modulo $s^k$ of:
  \begin{gather}\label{eq: base}
  e_1, \dots e_r,\dots , s^{k-1}e_1, \dots s^{k-1}e_r,\\
  v_{q+1}, \dots v_r,\dots , s^{k-1}v_{q+1}, \dots s^{k-1}v_r\nonumber\\
      s\sum_{j=1}^q\mu_{1j}f_j , \dots,s\sum_{j=1}^q\mu_{qj}f_j, \dots, s^{k-1}\sum_{j=1}^q\mu_{1j}f_j , \dots,s^{k-1}\sum_{j=1}^q\mu_{qj}f_j.\nonumber
  \end{gather}
It is easy to see that for every vanishing linear combination with  coefficients in $\K$ of the classes in \eqref{eq: base} the coefficients of the classes in the first two rows are trivial, and therefore
\[
\dim_{\mathbb C} \left(\left(W_1\otimes_A\!B_k\right)+ \left(W_2\!\otimes_A\!B_k\right) \right) =
k(2r-q)+r_k
\]
where $r_k$ is the dimension of the complex vector subspace of $V\otimes_A B_k$ 
generated by the classes in the last row of \eqref{eq: base}.

By the Grassman formula 
\[
q_k=\dim_{\mathbb C} \left(\left(W_1\otimes_A\!B_k\right) \cap  \left(W_2\!\otimes_A\!B_k\right) \right) =
kq-r_k.
\]
Consider the $q \times q$ matrix $M=\left(s \mu_{ij} \right)_{i=1,\ldots,q}^{j=1,\ldots,q}$ with entries in $A$.  Since $W_2$ is totally isotropic,
  for $1\le i,j\le q$ one has $0=Q(v_i, v_j)=s(\mu_{ij}+\mu_{ji})+s^2Q(z_i,z_j)$. By \eqref{eqn: zi}  $Q(z_i,z_j)=0$ and therefore $\mu_{ij}+\mu_{ji}=0$. So $M$ is skewsymmetric and $r_k$ is a non-decreasing sequence of even numbers by Lemma \ref{lem: skew}. 
 \end{proof}
  We recall also the following well known fact:
  
  \begin{lem} \label{lem: base change} Let $\Delta:=\Spec R$ be a smooth affine curve, let $\pi\colon X\to \Delta$ be a proper morphism with 
   1-dimensional fibers and let $F$ be a  coherent sheaf on $X$ flat over $\Delta$. Then there exists a two term complex $M^{\bullet}: 0\to M^0\to M^1\to 0$ of finitely generated locally free modules and an isomorphism of functors 
  $$H^p(X\times_{\Delta}\Spec B, F\otimes _RB)\cong H^p(M^{\bullet}\otimes _RB),\quad p\ge 0$$
  on the category of $R$-algebras $B$.
  \end{lem}
  \begin{proof}
  By the Theorem in \S 5 of Chapter II of \cite{mumford-abelian} there is a finite complex $K^{\bullet}: 0\to K^0\to\dots \to K^n\to 0$ of finitely generated projective $R$-modules and an isomorphism of functors 
  $$H^p(X\times_{\Delta}\Spec B, F\otimes _RB)\cong H^p(K^{\bullet}\otimes _RB),\quad p\ge 0$$ on the category of $R$-algebras.
Since $\Delta$ is a smooth curve the modules $K^{\bullet}$ and all their submodules are locally free. Set $M^0:=K^0$ and $M^1:= \ker(K^1\to K^2)$. The complex  $M^{\bullet}:=0\to M^0\to M^1\to 0$ has a natural map to $K^{\bullet} $ which is  a quasi-isomorphism, since the fibers of $\pi$ are $1$-dimensional. We conclude by applying  Lemma 2 in \S 5 of Chapter II of \cite{mumford-abelian} to the map of complexes of flat modules $M^{\bullet}\to K^{\bullet}$.
  \end{proof}
  \begin{cor}\label{cor: base change} In the assumptions of Lemma \ref{lem: base change}, if $B$ is an $R$-algebra then:
  \begin{enumerate}
  \item $H^1( X\times _{\Delta}B, F\otimes _RB)\cong H^1(X,F)\otimes _RB$
  \item if $H^1(X,F)=0$, then $H^0( X\times _{\Delta}B, F\otimes _RB)\cong H^0(X,F)\otimes _RB$
  \end{enumerate}
  \end{cor}
 \begin{proof}
Let $M^{\bullet}:=0\to M^0\to M^1\to 0$ be the complex given by Lemma \ref{lem: base change}. There is an exact sequence
\begin{equation}\label{eq: 4terms}
0\to H^0(X, F)\to M^0\to M^1\to H^1(X,F)\to 0
\end{equation}
and $H^i( X\times _{\Delta}B, F\otimes _RB)$ is the $i$-th cohomology of $M^{\bullet}\otimes _RB$. So (i) is a consequence of the fact that tensor product is right exact. 

If $H^1(X,F)=0$, then \eqref{eq: 4terms} gives a short exact sequence 
\begin{equation}\label{eq: 4->3terms}
0\to H^0(X, F)\to M^0\to M^1\to 0
\end{equation}
Since $M^1$ is locally free, and flatness is a local property (see \cite{TSP}  \href{https://stacks.math.columbia.edu/tag/00HT}{Lemma 00HT}) then $M^1$ is flat and therefore, by \href{https://stacks.math.columbia.edu/tag/00HL}{Lemma 00HL} of \cite{TSP},  \eqref{eq: 4->3terms} stays exact after tensoring with $B$. So $H^0(X,F)\otimes_R B \cong H^0(M^{\bullet}\otimes _RB) $, proving (ii).
 \end{proof}
  
\subsection{Proof of Thm. \ref{thm: infinitesimal}}
  
The proof follows the same lines as the proofs of \cite[\S 1]{Mum71}  and \cite[Theorem 1.10.(i)]{Har}.

The statement is local, hence we work near a fixed point $\bar t\in \Delta$ and we  denote by $C$ the fiber $C_{\bar t}$ and by $L$ the line bundle $\Lc_{\bar t}={\mathcal L}|_{C_{\bar t}}$. 
We pick  $p_1,\dots p_N$ smooth points of $C$ such that $h^0(C, L(-D))=h^1(C, L(D))=0$, where $D=\sum p_i$,  and consider the following diagram with exact rows and columns:
\begin{equation}\label{eq: D1}
\begin{tikzcd}
            &                                                & 0 \arrow[d]                                     & 0 \arrow[d]                    &   \\
0 \arrow[r] & L(-D) \arrow[r] \arrow[d, Rightarrow, no head] & L \arrow[r] \arrow[d]                           & L/L(-D) \arrow[r] \arrow[d]    & 0 \\
0 \arrow[r] & L(-D) \arrow[r]                                    & L(D) \arrow[r] \arrow[d]                        & L(D)/L(-D) \arrow[r] \arrow[d] & 0 \\
            &                                                & L(D)/L \arrow[d] \arrow[r,  Rightarrow, no head] & L(D)/L \arrow[d]               &   \\
            &                                                & 0                                               & 0                              &  
\end{tikzcd}
\end{equation}

Taking cohomology in \eqref{eq: D1} one gets another exact diagram:

\begin{equation}\label{eq: D2}
\begin{tikzcd}
            & 0 \arrow[d]                       & 0 \arrow[d]                         & 0 \arrow[d]                               &   \\
0 \arrow[r] & H^0(L) \arrow[r] \arrow[d]        & H^0(L/L(-D)) \arrow[r] \arrow[d]    & H^1(L(-D)) \arrow[d, Rightarrow, no head] &   \\
0 \arrow[r] & H^0(L(D)) \arrow[r] \arrow[d]     & H^0(L(D)/L(-D)) \arrow[r] \arrow[d] & H^1(L(-D)) \arrow[r]                      & 0 \\
            & H^0(L(D)/L) \arrow[r, Rightarrow, no head] & H^0(L(D)/L) \arrow[d]               &                                           &   \\
            &                                   & 0                                   &                                           &  
\end{tikzcd}
\end{equation}
Set $\ol V:=H^0(L(D)/L(-D))$, $\ol W_1:=H^0(L/L(-D))$, $\ol W_2:=H^0(L(D))$. Arguing as in \cite[\S 1]{Mum71}  and in the proof of \cite[Theorem 1.10.(i)]{Har} one shows that $\dim \ol V=2N$ and, using the isomorphism $2L\cong \omega_C$ constructs a non-degenerate bilinear form $\ol Q\colon \ol V\times \ol V\to\K$ such that $\ol W_1$ and $\ol W_2$ are maximal isotropic subspaces. Chasing through diagram \eqref{eq: D2} it is easy to check  that $H^0(L)$ can be identified with $\ol W_1\cap \ol W_2$.

The next step, still following  \cite[\S 1]{Mum71}  and  the proof of \cite[Theorem 1.10.(i)]{Har}, consists in giving a relative version of this construction.
\'Etale locally we may assume that $p_1,\dots p_N$ are cut out on $C$ by disjoint   sections $\sigma_1,\dots \sigma_N$ of $\pi$  contained in the smooth locus of $X$. Write $\mathcal D:=\sigma_1+\dots+\sigma_N$; up to shrinking $\Delta$ we may assume $\pi_*\Lc(-\Dc))=R^1\pi_*(\Lc(\Dc))=0$. 
 Then we have the following exact diagram,  the relative version  of \eqref{eq: D2}:
\begin{equation}\label{eq: D3}
\begin{tikzcd}
            & 0 \arrow[d]                       & 0 \arrow[d]                         & 0 \arrow[d]                               &   \\
0 \arrow[r] & \pi_*\Lc \arrow[r] \arrow[d]        & \pi_*(\Lc/\Lc(-\Dc)) \arrow[r] \arrow[d]    & R^1\pi_*(\Lc(-\Dc)) \arrow[d,  Rightarrow,no head] &   \\
0 \arrow[r] & \pi_*(\Lc(\Dc)) \arrow[r] \arrow[d]     & \pi_*(\Lc(\Dc)/\Lc(-\Dc)) \arrow[r] \arrow[d] & R^1\pi_*(\Lc(-\Dc)) \arrow[r]                      & 0 \\
            & \pi_*(\Lc(\Dc)/\Lc) \arrow[r, Rightarrow, no head] & \pi_*(\Lc(\Dc)/\Lc) \arrow[d]               &                                           &   \\
            &                                   & 0                                   &                                           &  
\end{tikzcd}
\end{equation}

We observe that, possibly up to shrinking $\Delta$ further,  we may assume: 	
\begin{itemize}
\item $R^1\pi_*\Lc(-\Dc)$ is free, since it  has constant rank, hence  the middle row of diagram \eqref{eq: D3} is split
\item   $\pi_*(\Lc(\Dc)/\Lc)$ is also free, since $\Dc\to \Delta$ is a finite flat map, hence  the middle column  of diagram \eqref{eq: D3} is also split
\item the sheaves $\mathcal V:= \pi_*(\Lc(\Dc)/\Lc(-\Dc))$, $\mathcal W_1:=\pi_*(\Lc/\Lc(-\Dc))$ and $\mathcal W_2:=\pi_*(\Lc(\Dc))$ are free
\end{itemize} 
Set $A:=\OO_{\Delta, \bar t}$, let $s\in A$ be a local parameter and let $B_k:=A/s^kA$, where $0<k\in \N$. Tensoring diagram \eqref{eq: D3} with $B_k$ gives:

\begin{equation}\label{eq: D4}
\begin{tikzcd}
            & 0 \arrow[d]                       & 0 \arrow[d]                         & 0 \arrow[d]                               &   \\
0 \arrow[r] &(\mathcal W_1\otimes B_k)\cap (\mathcal W_2\otimes B_k)\arrow[r] \arrow[d]        & \mathcal W_1\otimes B_k \arrow[r] \arrow[d]    & R^1\pi_*(\Lc(-\Dc))\otimes B_k \arrow[d, Rightarrow, no head] &   \\
0 \arrow[r] & \mathcal W_2\otimes B_k\arrow[r]      & \mathcal V \otimes B_k\arrow[r] \arrow[d]& R^1\pi_*(\Lc(-\Dc))\otimes B_k \arrow[r]                      & 0 \\
            &  & \pi_*(\Lc(\Dc)/\Lc) \otimes B_k\arrow[d]               &                                           &   \\
            &                                   & 0                                   &                                           &  
\end{tikzcd}
\end{equation}
 By the previous remarks diagram \eqref{eq: D4} is exact;  in addition all the direct image sheaves appearing in it satisfy cohomology and base change by Corollary \ref{cor: base change}. 
 So, setting $L_k:=\Lc|_{kC}$ and $\Dc_k:= \Dc|_{kC}$ it can be rewritten as:

\begin{equation}\label{eq: D5}
\begin{tikzcd}
            & 0 \arrow[d]                       & 0 \arrow[d]                         & 0 \arrow[d]                               &   \\
0 \arrow[r] &(\mathcal W_1\otimes B_k)\cap (\mathcal W_2\otimes B_k)\arrow[r] \arrow[d]        & H^0(kC,L_k/L_k(-\Dc_k))\arrow[r] \arrow[d]    & H^1(kC, L_k(-\Dc_k)) \arrow[d, Rightarrow, no head] &   \\
0 \arrow[r] & H^0(kC, L_k(\Dc_k)) \arrow[r]      & H^0(kC,L_k(\Dc_k)/L_k(-\Dc_k))\arrow[r] \arrow[d]& H^1(kC, L_k(-\Dc_k))  \arrow[r]                      & 0 \\
            &  & H^0(kC, L_k(\Dc_k)/L_k)\arrow[d]               &                                           &   \\
            &                                   & 0                                   &                                           &  
\end{tikzcd}
\end{equation}
The first row of \eqref{eq: D5} gives an identification $H^0(kC,L_k)\cong (\mathcal W_1\otimes B_k)\cap (\mathcal W_2\otimes B_k)$. 
Denote by $V$, resp.  $W_1, W_2$, the stalks at $\bar t$  of $\mathcal V$, resp. $\mathcal W_1,\mathcal W_2$, so that $\mathcal V\otimes B_k=V\otimes B_k$, resp. $\mathcal W_1\otimes B_k=W_1\otimes B_k$, $\mathcal W_2\otimes B_k=W_2\otimes B_k$. Our claim now follows by Lemma \ref{lem: aux} provided one extends the $\K$-bilinear form $\ol Q\colon \ol V\times \ol V\to \K$ to an $A$-bilinear form $Q\colon V\times V\to A$ such that $W_1$ and $W_2$ are isotropic subspaces. 
This boils down to being able to take residues of rational sections of $\Lc^{\otimes 2}$ along the components of $\Dc$ and it  can be done as in the proof of \cite[Theorem 1.10.(i)]{Har}. The surface $X$ is Gorenstein, since $\Delta$ is smooth and $\pi$ has Gorenstein fibers, and locally near $C$ there is an isomorphism $\Lc^{\otimes 2}\cong \omega_{\pi}$. In turn, $\omega_{\pi}$ restricts to the sheaf relative differentials on the smooth locus of $X$.

\subsection{Proof of Corollary \ref{cor: torsion}}
The statement is local, so we may assume that the torsion subsheaf $\mathcal R$ of $R^1\pi_*\Lc$ is supported at a single point $\bar t\in \Delta$. We write again $A:=\OO_{\Delta,\bar t}$ and $B_k:=A/s^kA$ for $s\in A$ a local parameter. Since $A$ is a DVR, there is a decomposition $\mathcal R=A/s^{r_1}A \oplus \dots \oplus A/s^{r_m}A$, where $1\le r_1\le \dots \le r_m$. For $j\in \N_{>0}$ we let $m_j$ be the number of indices $i$ such that $r_i\ge j$ (e.g., $m_1=m$). The statement is equivalent to showing that all the $m_j$ are even.

For $t\in \Delta$  denote by $C_t $ the fiber of $\pi$ over $t$ and set $C:=C_{\bar t}$, $L_k:=\Lc|_{kC}$. Denote by $q_0$ the rank of  $R^1\pi_*\Lc$; for $t\ne \bar t\in \Delta$, $q_0=h^1(C_t, \Lc|_{C_t})=h^0(C_t, \Lc|_{C_t})$, where the second equality follows by Riemann-Roch on $C_t$. 

By Corollary \ref{cor: base change}, $h^1(kC,L_k)=kq_0+m_1+\dots +m_k$. 
 On the other hand, Riemann-Roch on $kC$ gives $h^1(kC,L_k)=h^0(kC,L_k)$. By the constancy of the parity of theta characteristics in families (\cite[\S 1]{Mum71}  and \cite[Theorem 1.10.(i)]{Har}), $q_0$ and $h^0(C,L)$ have the same parity, so $m_1$ is even.
 Now Theorem \ref{thm: infinitesimal} implies that  $m_1+\dots +m_k$ is even for every $k\ge 2$, hence  all the $m_j$ are even.

\bigskip

\bigskip

\begin{minipage}{13.0cm}

\parbox[t]{5.5cm}{Margarida Mendes Lopes\\ Centro de An\'alise Matem\'atica, Geometria e Sistemas Din\^amicos,\\Departamento de  Matem\'atica\\
Instituto Superior T\'ecnico\\
Universidade de Lisboa\\
Av.~Rovisco Pais, 1\\
1049-001 Lisboa, Portugal\\
mmendeslopes@tecnico.ulisboa.pt
} \hfill
\parbox[t]{5.5cm}{Rita Pardini\\
Dipartimento di Matematica\\
Universit\`a di Pisa\\
Largo B. Pontecorvo, 5\\
56127 Pisa, Italy\\
rita.pardini@unipi.it}
\end{minipage}

\vskip1.0truecm

\parbox[t]{6.5cm}{Roberto Pignatelli\\
Dipartimento di Matematica\\
Universit\`a di Trento\\
via Sommarive 14\\
38123 Trento, Italy\\
roberto.pignatelli@unitn.it
 }

\end{document}